\documentclass[11pt,reqno]{amsart}
\usepackage{amsbsy,amsfonts,amsmath,amssymb,amscd,amsthm}
\usepackage{mathrsfs,float}
\usepackage[mathcal]{euscript}
\usepackage{enumitem,graphicx,epstopdf}
\usepackage[abs]{overpic}
\usepackage[a4paper,text={6.1in,8in},centering]{geometry}
\usepackage{pxfonts,epstopdf}
\usepackage[colorlinks=true,linkcolor=blue,citecolor=green]{hyperref}
\usepackage{srcltx}
\usepackage[margin=20pt,font=small,labelfont=bf,textfont=it]{caption}
\usepackage{graphicx}
\usepackage{subcaption}
\usepackage[latin1]{inputenc}

\hypersetup{linktocpage}

\small\normalsize
\theoremstyle{plain}
\newtheorem{mainthm}{Theorem}

\newtheorem{thm}{Theorem}[section]
\newtheorem{lem}[thm]{Lemma}

\newtheorem{prop}[thm]{Proposition}
\newtheorem*{thm*}{Theorem}

\newtheorem{defi}[thm]{Definition}

\theoremstyle{definition}
\newtheorem{rem}[thm]{Remark}

\newcommand{\eqdef}{\stackrel{\scriptscriptstyle\rm def}{=}}

\setlist[enumerate,1]{label=(\arabic*)}
\setlist[enumerate,2]{label=(\alph*)}

\begin{document}
%~\vspace{-0.5cm}
\title[Strange attractors]{Typical coexistence of infinitely many strange attractors
%in parameter families of diffeomorphisms
}

\begin{abstract}
We prove that the coexistence of infinitely  many prevalent
H\'enon-like phenomena is Kolmogorov typical in sectional
dissipative $C^{d,r}$-Berger domains of parameter families of
diffeomorphisms of dimension $m\geq 3$ for $d<r-1$. {Namely,}
we answer an old question posed by Colli in [\emph{Annales de l'Institut
Henri Poincare-Nonlinear Analysis}, 15, 539--580 (1998)] on
typicality \mbox{of the coexistence of infinitely many non-hyperbolic
strange attractors for $3\leq d<r-1$.}
\end{abstract}
\author[Barrientos]{Pablo G.~Barrientos}
\address{\centerline{Instituto de Matem\'atica e Estat\'istica, UFF}
    \centerline{Rua M\'ario Santos Braga s/n - Campus Valonguinhos, Niter\'oi,  Brazil}}
\email{pgbarrientos@id.uff.br}
\author[Rojas]{Juan David Rojas}
\address{\centerline{Universidad del Rosario,
Escuela de Ingenieria, Ciencia y Tecnologia}
    \centerline{Bogota,Colombia.}}
\email{juandavi.rojas@urosario.edu.co}

\maketitle  \thispagestyle{empty}
\section{Introduction}
Homoclinic bifurcations are one of the main mechanisms to create
complicated dynamical behavior in the evolution of parametric
families of discrete systems. A $C^r$-diffeomorphism has a
\emph{homoclinic tangency} if there is a pair of points $P$ and
$Q$ in the same transitive hyperbolic set such that the unstable
invariant manifold of ${P}$ and the stable invariant manifold of
$Q$ have a non-transverse intersection at a point $Y$. A
homoclinic tangency can be unfolded considering a $C^d$-family $(f_a)_a$ of $C^r$-diffeomorphisms parameterized by
$a\in \mathbb{I}^k$ with $f_{0}=f$, $\mathbb{I}=[-1,1]$, $k\geq 1$
and $d\leq r$. From the pioneering work of Newhouse~\cite{New70}, it
is well known that the set of $C^2$ surface diffeomorphisms
exhibiting homoclinic tangencies has a non-empty interior. See also
\cite{PV94,GST93a,Ro95,BD12,BR17} for higher dimensional dynamics.
In a \emph{Newhouse domain}, that is, in an open set of
diffeomorphisms (or in the corresponding parameters space) where the
dynamics with homoclinic tangencies associated with periodic
points are dense, generic diffeomorphisms exhibit coexistence of
infinitely many
sinks~\cite{New79,robinson1983bifurcation,GST93b,GST08}. This
result was coined under the name of \emph{Newhouse phenomenon}. But
also the unfolding of homoclinic tangencies in Newhouse domains
brings the presence and coexistence of more complicated chaotic
dynamics as (non-hyperbolic) strange attractors
(see~\cite{MV93,viana1993strange}). %That is, a compact invariant set having a dense
%orbit with at least one positive Lyapunov exponent and whose
%stable set has non-empty interior.
%
%\begin{defi}
{A \emph{strange attractor} of a transformation $g$ is an invariant compact set $\Lambda$ whose stable set $W^s(\Lambda)=\{x: d(g^n(x),\Lambda)\to 0\}$
has a non-empty interior, and there is $z\in \Lambda$ with dense orbit (in $\Lambda$) % $\{g^n(z): \, n\geq 0\}$
 displaying exponential
growth of the derivative, that is, $\|Dg^n(z)\|\geq e^{cn}$ for all $n\geq 0$ and some $c>0$.}
%\end{defi}

Obviously, strange attractors are always non-trivial (i.e., they
are not reduced to a periodic orbit). However, they could be still
hyperbolic, as for instance the Plykin attractor or the Smale
solenoid. One of the first examples of non-hyperbolic strange
attractors was given (numerically) by H\'enon~\cite{henon1976two}
for the two-parameter family given by
\[
H_{a,b}(x,y)=(1-ax^2+y,bx).
\]
The limit family ($b=0$) is given by the quadratic maps
$T_{a}(x)=1-ax^{2}$. Benedicks and Carleson proved
in~\cite{benedicks1985iterations} that there exists a positive
Lebesgue measure set of parameters $a \in (1,2]$ such that the
compact interval $[1-a,1]$ is a strange attractor of $T_{a}$.
These results, see also \cite{jakobson1981absolutely}, were key to
prove the existence of strange attractors for the family of
H\'{e}non maps in~\cite{benedicks1991dynamics}. It was quickly
observed that an extension of conclusions
in~\cite{benedicks1991dynamics} could be possible for some other
family $F_{a,b}$ whose family of limit maps was also the quadratic
family~$T_{a}$. Families of this type will be called
\textit{H\'enon-like families}. Namely, Mora and Viana
in~\cite{MV93} and~\cite{viana1993strange} showed that any
H\'enon-like family of diffeomorphisms has a set of parameters
with positive Lebesgue measure for which a strange
attractor is exhibited. Such families appear in generic unfoldings of
homoclinic tangencies associated with \emph{sectional dissipative
periodic points}~\cite{PT93,viana1993strange,GST93b,GST08}, that
is, periodic points which have the product of any pair of
multipliers less than one in absolute value.

{Sometimes the non-hyperbolic strange attractors found in H\'enon-like families are called \emph{H\'enon-like strange attractors}.
The lack of hyperbolicity of H\'enon-like strange attractors} prevents stability under
perturbations, and thus, the classical arguments (see~\cite{PT93})
to provide coexistence of infinitely many of such attractors do
not work. This difficulty was overcome by
Colli~\cite{colli1998infinitely} and Leal~\cite{leal2008high} who
proved that, in Newhouse domains associated with homoclinic
tangencies to sectional dissipative periodic points, there exists
a  dense set of diffeomorphisms (or corresponding
parameters in the parameter space) exhibiting the coexistence of
infinitely many non-hyperbolic strange attractors. These results
say nothing about the persistence of the coexistence of infinitely
many attractors. Recall that a parametric family $f=(f_a)_a$ of
dynamics exhibits \emph{persistently} a property $\mathscr{P}$ if
$\mathscr{P}$ is observed for $f_a$ in a set $E(f)$ of parameter values
$a$ with a positive Lebesgue measure. It was an old open question
due to Colli (see~\cite[pg.~542]{colli1998infinitely}) whether for
"most" or are least there exists a $k$-parameter family
$(f_{a})_{a}$ of diffeomorphisms which exhibits persistently the property of coexistence of infinitely many strange
attractors. %for values of the parameter $a$ in a positive Lebesgue measure set $E(f)\subset \mathbb{I}^{k}$.
The abundance of such
families must be understood in the sense of typicality introduced
by Kolmogorov (see~\cite{hunt2010prevalence}). That is, a property
$\mathscr{P}$ is called \emph{typical} (in the sense of
Kolmogorov) if there is a Baire (local) generic set of parameter
families of dynamics exhibiting the property $\mathscr{P}$
persistently with full Lebesgue measure.  {See Definition~\ref{def:tipical}
for a more precise statement of this important notion.}
On the other hand, Palis
claimed that the measure of the set of parameters $E(f)$ where infinitely many attractors coexist is generically zero for
families $f=(f_{a})_{a}$ of one-dimensional dynamics and surface
diffeomorphism~\cite{PT93,palis2000global}.

Pumari\~no and Rodr\'iguez in~\cite[Thm.~B]{Puma01} (see
also~\cite{Puma97}) provided a first example of a non-generic
family of dynamical systems with persistent coexistence of
infinitely many non-hyperbolic strange attractors. Although Palis'
conjecture remains open, some advances in the opposite
direction have been made by Berger in~\cite{Ber16,Ber17} for
families of surface endomorphisms (in fact, local diffeomorphisms)
and higher dimensional diffeomorphisms. Namely, Berger constructed
open sets $\mathcal{U}$ of $k$-parameter families of the above
described dynamics in the $C^{d,r}$-topology  such that residually in these open sets any
family exhibits simultaneously infinitely many hyperbolic periodic
attractors (sinks) \emph{for all} parameter value. {See~\S\ref{topology}
to see the definition of the $C^{d,r}$-topology.}
Mimicking previously introduced terminology, this result was coined
in~\cite{BR21} under the name of \emph{Berger phenomenon} and the
open sets $\mathcal{U}$ as \emph{Berger domains}. In fact,
in~\cite{BR21}, Raibekas and the first author of this work provide
the following more specific definition of such domains:

\begin{defi} \label{def:Berger-domain}
An open set $\mathscr{U}$ of $k$-parameter $C^d$-families of
$C^r$-diffeomorphisms is called \emph{$C^{d,r}$-Berger domain}
if there exists a dense subset
$\mathscr{D}$ of $\mathscr{U}$ such that for any $f=(f_a)_a \in
\mathscr{D}$ there is a covering of $\mathbb{I}^k$ by open balls
$J_i$
where %$f$ displays a persistent homoclinic tangency, i.e.,
$f_a$ has a homoclinic tangency $Y^i_{a}$
 associated with a
hyperbolic periodic point $Q^i_a$  depending $C^d$-continuously on the
parameter $a$ in $J_i$.
%V
%
%An open set $\mathcal{U}$ of  $k$-parameter $C^d$-families of
%$C^r$-diffeomorphisms is called \emph{$C^{d,r}$-Berger domain} of
%{persistent homoclinic tangencies} if there exists a dense subset
%$\mathcal{D}$ of $\mathcal{U}$ such that for any $f=(f_a)_a \in
%\mathcal{D}$ there is a covering of $\mathbb{I}^k$ by open balls
%$J_i$ having {the following property:} {there is
%a saddle periodic point $Q^i_a$ of $f_a$} having a homoclinic tangency {$Y^i_{a}$}
%which depends $C^d$-continuously on the parameter $a\in J_i$.
\end{defi}

{Roughly speaking, a Berger domain is the equivalent of a Newhouse domain for parametric families. {As in a Newhouse domain (free parameter case), a Berger domain has a dense subset of families $f=(f_a)_a$ where each diffeomorphism $f_a$ in the family exhibits a homoclinic tangency $Y_a$.
But now, for every $a_0$, the tangency $Y_{a_0}$ of $f_{a_0}$ is unfolding degenerately\footnote{The notion of \emph{degenerate unfolding} is introduced in~\cite[Def.~3.8]{Ber16} under the name of \emph{paratangency}. For a more precise definition and contextualization see~\cite[Sec.~1.4]{BR20}.} inside the family $f=(f_a)_a$ in the sense that persists (as a homoclinic tangency) along an open set of parameters $J$ containing $a_0$, i.e., varies continuously with respect to $a$ in $J$.}  For more details and a deeper description of the notion of Berger domain, see~\cite[Sec.~1.2]{BR21}.}

%This definition appears implicitly in the constructions of Berger
%in~\cite{Ber16,Ber17}.

{In~\cite{Ber16,Ber17},  Berger} constructed $C^{d,r}$-Berger
domains of persistent homoclinic tangencies associated with
sectional dissipative periodic points {$Q^i_a$} with $d\leq r$ for
endomorphisms in dimension two and diffeomorphisms in higher
dimension. New and different examples of $C^{d,r}$-Berger domains
of $k$-parametric families of diffeomorphisms in dimension $m\geq
3$ for $d<r-1$ were obtained also in~\cite{BR21} as a consequence
of~\cite{BR20}.  The
coexistence of infinitely many smooth attracting invariant circles
was shown also $C^{d,r}$-Kolmogorov typical {in these new examples of Berger domains.}  In this paper, we
will answer the previously mentioned question posed by Colli by showing that the
coexistence of infinitely many non-hyperbolic strange attractors
is typical~{for $k$-parametric families} in the sense of Kolmogorov {in sectional dissipative Berger domains}:

\begin{mainthm} \label{mainthmA}
Let $\mathcal{U}$ be a $C^{d,r}$-Berger domain of $k$-parameter
families of diffeomorphisms of dimension $m\geq 3$ with $3\leq d<
r-1$ and $k\geq 1$ associated with sectional dissipative periodic
points. Then there exists a residual set $\mathcal{R}$ of
$\mathcal{U}$ such that for every family $(f_a)_a \in \mathcal{R}$
the diffeomorphism $f_a$ exhibits infinitely many non-hyperbolic
strange attractors for Lebesgue almost every $a\in \mathbb{I}^k$.
\end{mainthm}

The proof of the above theorem is strongly based on the fact that the coexistence of non-hyperbolic strange attractors is a prevalent phenomenon for H\'enon-like families.
To be more precise, we need some definitions.

Let us denote by $\Phi=(\Phi_M)_{M}$ the parabola
family given by
\begin{equation} \label{eq:parabola-family}
\Phi_M(x,y)=(0,M-y^2) \quad \text{where \ \ $(x,y)\in
[-3,3]^{m-1}\times [-3,3]$ \ \ and \ \ $M\in [1,2]$}.
\end{equation}
A phenomenon $\mathscr{P}$ is a fact or property that is observed to exist or to occur in a dynamics.  Examples of phenomena for the quadratic map are the existence of sinks, saddle-nodes, flip bifurcation, homoclinic tangencies, and strange attractors among others.

\begin{defi} \label{def:prevalent-property}
A phenomenon $\mathscr{P}$ of a dynamics is said to be
\emph{$C^s$-prevalent for H\'enon-like families} if there {exists
 $0<c\leq 1$} such that any $C^{d,r}$-family
$\varphi=(\varphi_M)_M$ of diffeomorphisms with $s\leq d \leq r$
which is {sufficiently} $C^s$-close to the parabola family
$\Phi=(\Phi_M)_M$ satisfies $\mathscr{P}$ for any $M$ in a subset
of parameters of $(1,2)$ with Lebesgue measure at least $c$.
\end{defi}

An example of $C^s$-prevalent phenomenon for H\'enon-like families
is the existence of hyperbolic attractors~\cite{PT93,GST08} (with $s=2$).
As explained in~\cite[Sec.~5]{colli1998infinitely}, \cite[Sec.~3.2]{leal2008high} and notify in~\cite[p.~52]{DRV:96}, it follows
from the proof of the results in~\cite{MV93,viana1993strange} that the existence of non-hyperbolic strange
attractors is also a $C^s$-prevalent phenomenon for H\'enon-like families (with $s=3$). {However, in this case, we need to be careful with the meaning of sufficiently close to the parabola family (see conditions~(QL) and (SD) in~\cite{viana1993strange}\footnote{The referee pointed out that similarly to Mora-Viana's proof, Viana's proof needed one extra assumption on  distorsion bound of the determinant of the family. This property seems nonetheless not necessary in view of the
alternative proof of~\cite{berger2019strong}}). The family $\varphi$ in the above definition needs to belong to an open set of $C^3$-perturbations of the family of parabolas, which occur in the renormalization scheme of the unfolding of some homoclinic tangencies.}

%\enlargethispage{-0.7cm}

%Thus, Theorem~\ref{mainthmA} is actually a corollary of the following theorem. {First recall that a dynamics $g$ exhibits the \emph{coexistence of infinitely many %phenomena $\mathscr{P}$} if there is sequence of mutually independent phenomena $\mathscr{P}_n$ of type $\mathscr{P}$ such that $g$ exhibits $\mathscr{P}_n$ for all %$n\geq 1$.

%\begin{mainthm} \label{mainthmB}
%Under the assumption of Theorem~\ref{mainthmA},
%%Let $\mathcal{U}$ be a $C^{d,r}$-Berger domain of $k$-parameter families of diffeomorphisms with $d< r-1$ and $k\geq 1$ associated with sectional dissipative periodic points.
%if $\mathscr{P}$ is a $C^s$-prevalent phenomenon for H\'enon-like
%families with $s\leq d$, then there exists a residual set
%$\mathcal{R}$ of $\mathcal{U}$ such that for every family $(f_a)_a
%\in \mathcal{R}$ the diffeomorphism $f_a$ exhibits the coexistence
%of infinitely many phenomena $\mathscr{P}$ for Lebesgue almost
%every~$a\in \mathbb{I}^k$.
%\end{mainthm}
%\vspace{-0.1cm}

Many of the ideas in this paper come from the Ph.D.~thesis of
the second author~\cite{R17} where similar results were obtained in
the case of endomorphisms on surfaces. Note that the main techniques used here to prove Theorem~\ref{mainthmA} are essentially different from those introduced by Berger in~\cite{Ber16,Ber17}. To describe the differences, let us first explain briefly
the strategy developed by Berger to get his results.

As we mentioned earlier, a family of diffeomorphisms $f=(f_a)_a$ that unfolds a homoclinic tangency of $f_0$   exhibits many different phenomena. Some of these phenomena are robust in the sense that they appear for open sets of parameters arbitrarily close to $a=0$. For instance, we have the existence of (hyperbolic) sinks. But they are quickly destroyed by a small perturbation of the dynamics. In other words, the lifetime of these sinks (the size of the open set of parameters where there exists a well-defined continuation) is very small. The persistent homoclinic tangency $Y^i_a$ associated with the saddle $Q^i_a$ that appears in Definition~\ref{def:Berger-domain} allows Berger, by means of an arbitrarily small perturbation of the family $f$, to create a sink with a large timelife. Indeed, the created sink has a well-defined continuation for all parameters $a\in J_i$. Since a sink is robust (i.e.,~it has a continuation for
nearby systems), it is obtained that any family sufficiently close has a sink for all parameters $a\in J_i$. Now, Berger's phenomenon is proved by well-known arguments used to prove Newhouse's phenomenon. %it is proved the Berger phenomenon.

To address a similar result but involving non-hyperbolic phenomena (as H\'enon-like attractors) we have to tackle the fact that these phenomena are not robust. In this direction, we replace the notion of topological robustness by the notion of prevalent for H\'enon-like families. By introducing a new parameter $\mu$, we consider a family $g=(g_{a,\mu})_{a,\mu}$ which unfolds the homoclinic tangency $Y^i_a$ of $f_a$ generically with respect to $\mu$ and such that $g_{a,0}=f_a$.  With the help of the rescaling lemmas in~\cite{GST08}, we can get a curve $C_n=\{(a,\mu_n(a))\}$ in the $(a,\mu)$-parameter space arbitrarily close to $\{(a,0)\}$ such that $g|_{C_n}=(g_{a,\mu_n(a)})_a$ is a family close to $f$ and it is also a H\'enon-like family (after renormalization in $J_i$). Now, the prevalence allows us to conclude Theorem~\ref{mainthmA} similarly as Berger did to obtain his result.

In the next section, we
introduce formally the set of families of diffeomorphisms that we
are considering and the $C^{d,r}$-topology. After that, we will
prove Theorem~\ref{mainthmA}.

%\enlargethispage{\baselineskip}
%\enlargethispage{-1cm}
%\vspace{-0.1cm}
\section{Typical coexistence of infinitely many prevalent phenomena}

\subsection{Topology of families of diffeomorphisms}
\label{topology} We introduce the topology of the set of families.
To do this, set $\mathbb{I}=[-1,1]$.  Given $0<d \leq r\leq
\infty$, $k\geq 1$ and a compact manifolds $\mathcal{M}$, we
denote by $C^{d,r}(\mathbb{I}^k,\mathcal{M})$ the space of
$k$-parameter $C^d$-families $f=(f_a)_a$ of $C^r$-diffeomorphisms
$f_a$ of $\mathcal{M}$ parameterized by $a$ in an open
neighborhood of $\mathbb{I}^k$ such that
\begin{equation*}
  \partial^i_a \partial^j_x f_a(x) \ \ \text{exists continuously
  for all $0\leq i\leq d$, \ \   $0\leq i+j\leq r$  \ \  and \ \
   $(a,x)\in \mathbb{I}^k\times \mathcal{M}.$}
\end{equation*}
We endow this space with the topology given by the $C^{d,r}$-norm
given by
$$
\|f\|_{{C}^{d,r}}=\max\{\sup \|\partial^i_a\partial_x^j f_a (x)\|:
\, 0\leq i \leq d, \  0\leq i + j \leq r\} \quad \text{where \
$f=(f_a)_a \in {C}^{d,r}(\mathbb{I}^k,\mathcal{M})$.}
$$
If $d = r$, we will say that the family is of class $C^r$. Note
that a family $f=(f_a)_a$ is of class $C^r$, if and only if the
map $(a, x) \mapsto f_a(x)$ is of class $C^r$.

\begin{defi} \label{def:tipical} Fix $0<d \leq r\leq
\infty$, $k\geq 1$ and {an open set $\mathcal{U}$ of $C^{d,r}(\mathbb{I}^k,\mathcal{M})$.}
A property
$\mathscr{P}$ is said to be {\emph{$C^{d,r}$-Kolmogorov typical in $\mathcal{U}$}}  if there is a residual set $\mathcal{R}$ of $\mathcal{U}$ such that for every $f=(f_a)_a\in \mathcal{R}$  there is a set $E$ of full Lebesgue measure in~$\mathbb{I}^k$ such that $f_a$ exhibits the property $\mathscr{P}$ for all $a\in E$.
\end{defi}
{
Sometimes, to emphasize the notion of locally genericity without mentioning the open set~$\mathcal{U}$, $\mathscr{P}$ is simply said to be  \emph{locally Kolmogorov typical} in parametric families.
}
\subsection{Proof of Theorem~\ref{mainthmA}} From now on, fix a manifold $\mathcal{M}$ of dimension $m\geq 3$, $k\geq 1$ and \mbox{{$0<3\leq s\leq d<r-1$.}}
Consider a Berger domain $\mathcal{U}\subset
C^{d,r}(\mathbb{I}^k,\mathcal{M})$  associated with
sectional dissipative hyperbolic periodic points. Let
$\mathscr{P}$ be the property "existence of a non-hyperbolic strange attractor"\footnote{The proof also works for other prevalent phenomenon, such as the
 existence of a hyperbolic attractor.}.
 We will prove that the
coexistence of infinitely many phenomena $\mathscr{P}$ is
Kolmogorov typical in $\mathcal{U}$. First, we need to introduce an
important definition:
\begin{defi} \label{def:Henon-like-after}
Let $f=(f_a)_a$ be a family in
$C^{d,r}(\mathbb{I}^k,\mathcal{M})$  and fix
$\alpha<\beta$, $n\geq 1$ {and $\rho>0$.} The family $f$ is said to be a
\emph{$\rho$-$C^s$-H\'enon-like family after renormalization of period
$n$ in $I=(\alpha,\beta)^k$} if, for each $\bar{a}\in
[\alpha,\beta]^{k-1}$, there is a one-parameter family
$R_{\bar{a}}=(R_{\bar{a},b})_b$ of smooth transformations
$R_{\bar{a},b}:[-3,3]^m\to \mathcal{M}$ with $b\in [\alpha,\beta]$
such that the family $F=(F_{M})_M$ given by
$$
   F^{}_{M}\eqdef R_{a(M)}^{-1}\circ f^n_{a(M)} \circ R^{}_{a(M)}
$$
where
$$
a(M)=(\bar{a},b(M))\in I
 \ \ \text{with} \ \
b(M)=(\beta-\alpha)M+2\alpha-\beta \ \ \text{for} \ \ M\in [1,2]
$$
is $\rho$-$C^s$-close to the parabola family $\Phi=(\Phi_M)_M$
given in~\eqref{eq:parabola-family}.
\end{defi}

We observe that to be a $\rho$-$C^s$-H\'enon-like family  after
renormalization of period $n$ in $I$ is an open property in the
$C^{d,r}$-topology of parametric families.

{%\color{blue}
\begin{rem} \label{rem0} Here we will explain in more detail the prevalence of $\mathscr{P}$ that is required.
We will say that $f=(f_a)_a \in C^{d,r}(\mathbb{I}^k,\mathcal{M})$ is a \emph{$C^s$-H\'enon-like family after renormalization in $I\subset \mathbb{I}^k$}
if there are  $\rho_\ell\to 0^+$ and $n_\ell \to \infty$ such that $f$ is a $\rho_\ell$-$C^s$-H\'enon-like family after renormalization of period
$n_\ell$ in $I$ for all $\ell\geq 1$.  Such a family is obtained by a renormalization scheme in the unfolding of some homoclinic tangencies.
In particular, the renormalized families $F_{\ell}=(F_{\ell,M})_M$ converge to $\Phi$ as $\ell\to \infty$ in the $C^s$-topology. Therefore, for $\ell$ large enough, they are quadratic-like families as defined in~\cite[see~(QL) and p.~21]{viana1993strange} and thus $F_{\ell}$ has a strange attractor for a positive Lebesgue measure set $J_\ell \subset (1,2)$ of parameters~$M$.\footnote{Smooth linearization near the saddle is not necessary for this conclusion. See~\cite{Ro95} and~\cite{GST08} where the linearizability conditions were removed in the development of the renormalization scheme in the unfolding of generic homoclinic tangencies.} %Compare with~\cite[p.~53]{DRV:96}.}
Moreover, paraphrasing~\cite[pp.~52 and 54]{DRV:96}, actually the proof in~\cite{viana1993strange} provides a uniform lower bound for the measure of theses sets $J_\ell$ on a neighborhood of the family $\Phi$. We conclude that there are $c_0>0$ and $\ell_0\geq 1$ such that for every $\ell\geq \ell_0$, the set $J_\ell$  has Lebesgue measure $|J_\ell|\geq c_0$. Then $f^{n_{\ell}}_a$ has a strange attractor for any $a$ in a subset $J^*_\ell \subset I$ of parameters with
Lebesgue measure at least $ c\cdot |I|$ for some $c>0$ independent of $\ell$. Compare with~\cite[Sec.~5]{colli1998infinitely} and \cite[Sec.~3.2]{leal2008high}.
Therefore, following Definition~\ref{def:prevalent-property} and in order to emphasize the parameters $c>0$, we will say that $\mathscr{P}$ is a  \emph{$c$-prevalent phenomenon for $C^s$-H\'enon-like families after renormalization}.
\end{rem}
}

%Furthermore, if
%$f=(f_a)_a \in C^{d,r}(\mathbb{I}^k,\mathcal{M})$ is a
%$\rho$-$C^s$-H\'enon-like family after renormalization of period
%$n$ in $I=(\alpha,\beta)^k$  and $\mathscr{P}$ is a $c$-prevalent
%phenomenon for $\rho$-$C^s$-H\'enon-like families, then $f^n_a$
%exhibits $\mathscr{P}$ for any $a$ in a subset of parameters with
%Lebesgue measure $ c\cdot |I|$.  Here $|I|=(\beta-\alpha)^k$
%denotes the volume of $I$. To simplify notation, we will refer to
%the Cartesian product of $k$ open intervals with the same length
%as an \emph{open ball} (in the supremum norm in $\mathbb{R}^k$).
%
%
%
%Recall that this property is a
%$C^s$-prevalent phenomenon for H\'enon-like
%families (with $s=3$). Hence, there exist $\rho>0$ and $c>0$
%such that Definition~\ref{def:prevalent-property} holds. From this point on, in order to emphasize these parameters, we will say that
%$\mathscr{P}$ is a \emph{$c$-prevalent phenomenon for
%$\rho$-$C^s$-H\'enon-like families}.

%\subsubsection{H\'enon-like families after renormalization}
Recall that $\mathcal{D}$ denotes the dense set in $\mathcal{U}$
provided by Definition~\ref{def:Berger-domain}.  We also assume that the period of the periodic points $Q_a^i$ that appears in this definition is equal to one. This does not affect the argument of the proof, and it will suppose a considerable simplification in the notation and the statement of the next results. Also, to simplify notation, we will refer to
the Cartesian product of $k$ open intervals with the same length
as an \emph{open ball} (in the supremum norm in $\mathbb{R}^k$).

\begin{prop} \label{main-lema}
For any $f=(f_a)_a  \in \mathcal{D}$
one finds $\alpha_0=\alpha_0(f)>0$ such that for every $\epsilon>0$, $\rho>0$ and $0<\alpha \leq \alpha_0$, there are $n_0=n_0(\epsilon,\rho,f,\alpha)\in \mathbb{N}$ and a finite collection $\{I_j\}_j$ of pairwise disjoint open balls $I_j=I_j(f,\alpha)$  of $\mathbb{I}^k$
with $$|I_j| \leq \alpha |\mathbb{I}^k|, \quad |\mathbb{I}^k\setminus \cup_j I_j |\leq  \alpha  |\mathbb{I}^k|$$
%$$|I_j|= \alpha^k |\mathbb{I}^k|, \quad |\mathbb{I}^k\setminus \cup_j I_j |\leq \alpha (1+\alpha)^{-1} |\mathbb{I}^k|$$
%$$|I|=(2\alpha)^k \quad \text{and} \quad |\mathbb{I}^k\setminus \cup I_j |\leq 1/(1-\alpha)$$
 and the
following property:

for every $n \geq n_0$ there is an $\epsilon$-close family
$g=(g_a)_a$ to $f=(f_a)_a$ in the $C^{d,r}$-topology such that $g$
is a $\rho$-$C^s$-H\'enon-like family after renormalization of
period $n$ in any $I_j$.
\end{prop}

\begin{rem} \label{rem:geometrically-independence} By construction, the renormalization after period $n$ is done on the first return map $T_n=f_a^n$ associated with a homoclinic tangency of $f_a$ in Definition~\ref{def:Berger-domain}. Actually, $T_n$ is defined on a box $\sigma_n=\sigma_n(a)$ {as the first return to a tower of boxes $\cup_{n'} \sigma_{n'}$   with $\sigma_n\cap \sigma_{n'}=\emptyset$ if $n\not = n'$.} Thus, the phenomena $\mathscr{P}$ after renormalization of periods $n$ and $n'$ with $n\not=n'$ are geometrically independent. That is, if  $g_a$ has a phenomenon $\mathscr{P}$ after renormalization of period $n$ and $n'$ with $n\not = n'$  where $g=(g_a)_a$ is provided in Proposition~\ref{main-lema}, then we have a pair of periodic attractors $\Lambda_a$ and $\Lambda'_a$  of $g_a$ of minimal period, i.e., fixed attractors  of $T_n$ and~$T_{n'}$. In particular, their orbits are pairwise disjointed, and the attractors are distinct.
\end{rem}

Before proving the above proposition, we will conclude
Theorem~\ref{mainthmA}:

\begin{thm} For every $m\in \mathbb{N}$ and $\rho>0$, there exists an open and dense set
$\mathcal{O}_m=\mathcal{O}_m(\rho)$ in $\mathcal{U}$ such that it
holds the following:
\begin{enumerate}[label={}, rightmargin=2em, leftmargin=2em]
\item For any family $g=(g_a)_a$ in $\mathcal{O}_m$ and each
$\ell=1,\dots,m$, there exist $\alpha_\ell=2^{-\ell}\alpha_0>0$ (with $\alpha_0=\alpha_0(g)$),
 a positive integer $n_\ell$ (with
$n_1<\dots<n_m$) and a finite collection $\{I_{\ell,j}\}_j$ of pairwise disjoint open balls $I_{\ell,j}$ of $\mathbb{I}^k$ with
%$|I_{\ell,j}|= (\alpha_\ell)^k|\mathbb{I}^k|$, $|\mathbb{I}^k\setminus \cup_j I_{\ell,j}| \leq \alpha_\ell(1+\alpha_\ell)^{-1} |\mathbb{I}^k|$
$|I_{\ell,j}| \leq \alpha_\ell |\mathbb{I}^k|$ and $|\mathbb{I}^k\setminus \cup_j I_{\ell,j}| \leq \alpha_\ell |\mathbb{I}^k|$
such that $g$ is a
$\rho$-$C^s$-H\'enon-like family after renormalization of period
$n_\ell$ in any  $I_{\ell,j}$. %for all $\ell=1,\dots,m$.
\end{enumerate}
\noindent Moreover, there is a residual subset $\mathcal{R}$ of
$\mathcal{U}$ such that any family $g=(g_a)_a \in \mathcal{R}$
satisfies that $g_a$ has the coexistence of infinitely many
phenomenon $\mathscr{P}$ for Lebesgue almost every $a\in
\mathbb{I}^k$.
\end{thm}
\begin{proof}  First of all consider
the sequence $\epsilon_i=1/i$ for $i\geq 1$. We will prove the
result by induction. To do this, we are going to construct
$\mathcal{O}_m$ for $m=1$.

By Proposition~\ref{main-lema}, for each $f=(f_a)_a$ in $\mathcal{D}$ one finds $\alpha_0=\alpha_0(f)>0$ such that for $\alpha_1=2^{-1}\alpha_0$ and every
$\rho>0$, there are $n_0(i)=n_0(\epsilon_i,\rho,f,\alpha_1)\in \mathbb{N}$
and a finite collection $\{I_{1,j}\}_j$ of pairwise disjoint open
balls $I_{1,j}=I_{1,j}(f,\alpha_1)$ of $\mathbb{I}^k$ with $|I_{1,j}|\leq\alpha_1 |\mathbb{I}^k|$, $|\mathbb{I}^k\setminus \cup_j I_{1,j}|\leq\alpha_1|\mathbb{I}^k|$ and the
following property: For any $n \geq n_0(i)$, we get an $\epsilon_i$-close
family $g_i=(g_{i,a})_a$ to $f$ such that $g_i$ is a
$\rho$-$C^s$-H\'enon-like family after renormalization of period
$n$ in any  $I_{1,j}$ for all $i\geq 1$. Since this property
persists under perturbations, we have a sequence
$\{\mathcal{O}_1(f,\epsilon_i,\rho)\}_i$ of open sets
$\mathcal{O}_1(f,\epsilon_i,\rho)$ converging to $f$ where the
same conclusion holds for any family in these open sets. By taking
the union of all these open sets for any $f$ in $\mathcal{D}$ and
$\epsilon_i>0$ for $i\geq 1$, we get an open and dense set
$\mathcal{O}_1=\mathcal{O}_1(\rho)$ in $\mathcal{U}$ where for any
$g=(g_a)_a\in \mathcal{O}_1$  there exist $\alpha_0>0$,  $n_1\in \mathbb{N}$ and a finite collection $\{I_{1,j}\}_j$ of pairwise disjoint open balls $I_{1,j}$ of $\mathbb{I}^k$ with
$|I_{1,j}| \leq \alpha_1 |\mathbb{I}^k|$ and $|\mathbb{I}^k\setminus \cup_j I_{1,j}| \leq \alpha_1 |\mathbb{I}^k|$ where $\alpha_1=2^{-1}\alpha_0$  such that $g$
is a $\rho$-$C^s$-H\'enon-like family after renormalization of
period $n_1$ in any $I_{1,j}$.

Now we will assume $\mathcal{O}_m=\mathcal{O}_m(\rho)$ constructed
and we will show how to obtain $\mathcal{O}_{m+1}$. Since
$\mathcal{O}_m$ is an open and dense set in $\mathcal{U}$, we can
start by taking $f=(f_a)_a \in \mathcal{O}_m\cap \mathcal{D}$.
Hence, there is $\alpha_0=\alpha_0(f)>0$ such that for each $\ell=1,\dots,m$,
there is a positive integer $n_\ell$ (with
$n_1<\dots<n_m$) and a finite collection $\{I_{\ell,j}\}_j$ of pairwise disjoint open balls $I_{\ell,j}$ of $\mathbb{I}^k$ with
$|I_{\ell,j}| \leq \alpha_\ell |\mathbb{I}^k|$ and $|\mathbb{I}^k\setminus \cup_j I_{\ell,j}| \leq \alpha_\ell |\mathbb{I}^k|$ where $\alpha_\ell=2^{-\ell}\alpha_0>0$,
satisfying that $f$ is a $\rho$-$C^s$-H\'enon-like family after
renormalization of period $n_\ell$ in any $I_{\ell,j}$. As before,
from the robustness of this property, there exists
$\epsilon'=\epsilon'(f)>0$ such that any $\epsilon'$-close family
$g=(g_a)_a$ to $f$ still is a $\rho$-$C^s$-H\'enon-like family
after renormalization with respect to the same periods and in the
same open balls. Then, for $\alpha_{m+1}=2^{-(m+1)}\alpha_0$ and any $\epsilon_i<\epsilon'/2$, we can
apply Proposition~\ref{main-lema} finding
$n_0(i)=n_0(\epsilon_i,\rho,f,\alpha_{m+1})\in \mathbb{N}$ and a finite collection
$\{I_{m+1,j}\}_j$ of pairwise disjoint open balls $I_{m+1,j}=I_{m+1,j}(f,\alpha_{m+1})$ of $\mathbb{I}^k$ with $|I_{m+1,j}|\leq\alpha_{m+1} |\mathbb{I}^k|$ and $|\mathbb{I}^k\setminus \cup_j I_{m+1,j}| \leq \alpha_{m+1} |\mathbb{I}^k|$.
Moreover, by taking an integer $n_{m+1}>\max\{n_0(i),n_m\}$, we
get an $\epsilon_i$-perturbation $g_i=(g_{i,a})_a$ of $f$ such
that $g_i$ is a $\rho$-$C^s$-H\'enon-like family after
renormalization of period $n_{m+1}$ in any $I_{m+1,j}$ and $i\geq
1$. Hence, by the robustness as before, we have a sequence
$\{\mathcal{O}_{m+1}(f,\epsilon_i,\rho)\}_i$ of open sets
$\mathcal{O}_{m+1}(f,\epsilon_i,\rho)\subset \mathcal{O}_m$
converging to $f$ where the same conclusion holds for any family
in these open sets. Taking the union of all these open sets for
any $f\in \mathcal{O}_m\cap \mathcal{D}$ and
$\epsilon_i<\epsilon'(f)$, we get an open and dense set
$\mathcal{O}_{m+1}=\mathcal{O}_{m+1}(\rho)$ in $\mathcal{U}$. In
addition, for any
$g=(g_a)_a\in \mathcal{O}_{m+1}$  there exist $\alpha_0>0$,  positive integers $n_1<\dots <n_{m+1}$  and, for each $\ell=1,\dots,m+1$, a finite collection $\{I_{\ell,j}\}_j$ of pairwise disjoint open balls $I_{\ell,j}$ of $\mathbb{I}^k$ with
$|I_{\ell,j}| \leq \alpha_\ell |\mathbb{I}^k|$ and $|\mathbb{I}^k\setminus \cup_j I_{\ell,j}| \leq \alpha_\ell |\mathbb{I}^k|$ where $\alpha_\ell=2^{-\ell}\alpha_0$   such that $g$
is a $\rho$-$C^s$-H\'enon-like family after renormalization of
period $n_\ell$ in any~$I_{\ell,j}$.

To conclude the proof of the theorem, we need to prove that the
coexistence of infinitely many phenomena $\mathscr{P}$ is
Kolmogorov typical in $\mathcal{U}$. To do this, let $0<c\leq 1$ be the constant that appears {in Remark~\ref{rem0}.}
%and $\rho>0$ be the constant that appears in
%Definition~\ref{def:prevalent-property}.
Consider now the residual
set of $\mathcal{R}$ of $\mathcal{U}$ given by the intersection of
{$\mathcal{O}_{m,\ell}=\mathcal{O}_m(\rho_\ell)$ for all $m,\ell \in\mathbb{N}$ where $\rho_\ell\to 0^+$ as $\ell\to \infty$.}
Hence, any $g=(g_a)_a\in \mathcal{R}$ belongs to $\mathcal{O}_{m,\ell}$
for all $m,\ell\in \mathbb{N}$. Thus,  we find $\alpha_0=\alpha_0(g)>0$,  {and by a diagonal argument}, a strictly increasing
sequence of positive integers $(n_\ell)_\ell$ and collections
$
\{I_{\ell,j}\}_j$ of finitely many pairwise disjoint open balls $I_{\ell,j}$ in $\mathbb{I}^k$
with $|I_{\ell,j}| \leq \alpha_\ell |\mathbb{I}^k|$ and $|\mathbb{I}^k\setminus \cup_j I_{\ell,j}| \leq \alpha_\ell |\mathbb{I}^k|$ where $\alpha_\ell=2^{-\ell}\alpha_0$
such that $g$ is {$\rho_\ell$}-$C^s$-H\'enon-like renormalizable after
period $n_\ell$ in any $I_{\ell,j}$ for all $\ell\geq 1$. Let
$J^*_{\ell,j} \subset I_{\ell,j}$ be the set  of parameters where
a phenomenon $\mathscr{P}$ holds after renormalization of period
$n_\ell$. Notice that since $\mathscr{P}$ is $c$-prevalent
phenomenon for  {$C^s$-H\'enon-like families after renormalization}, we obtain  that
$J^*_{\ell,j}$  has at least Lebesgue measure
\begin{equation}\label{eq:c}
  |J^*_{\ell,j}|\geq c\cdot |I_{\ell,j}| \ \ \text{ for all $\ell$ {large enough}}.
\end{equation}
%Recall that here $|I|$ denotes the volume of the open $k$-interval $I$.
Let $A$ be a measurable set in $\mathbb{I}^k$ with $|A|>0$. By Lebesgue's density theorem, we have a density point $a_0\in A$. That is, $$\lim_{\varepsilon\to 0} \frac{|A\cap B_\varepsilon(a_0)|}{|B_{\varepsilon}(a_0)|}=1$$
where $B_\varepsilon(a_0)$ denotes the open ball in $\mathbb{I}^k$ centered at $a_0$ with radius $\varepsilon$ (in the usual norm in~$\mathbb{R}^k$). Hence, for a given $\delta>0$, there is $\varepsilon>0$ such that
\begin{equation}\label{eq:AcapB}
  |A\cap B_{\varepsilon}({a_0})| \geq (1-\delta) |B_\varepsilon(a_0)|. % \quad \text{for any $\varepsilon_0\geq \varepsilon>0$}.
\end{equation}
Since  $|I_{\ell,j}| \leq \alpha_\ell |\mathbb{I}^k|  \to 0$  and $|\mathbb{I}^k\setminus \cup_j I_{\ell,j}|  \leq\alpha_\ell |\mathbb{I}^k| \to 0$ as $\ell \to \infty$, for any $\ell$ large enough, we can extract a subcollection $\{I_{\ell,j_i}\}_i$ of
$\{I_{\ell,j}\}_j$ with $I_{\ell,j_i} \subset B_\varepsilon(a_0)$ (pairwise disjoints) and $|\cup_i I_{\ell,j_i}|\geq |B_\varepsilon(a_0)|/3$. Thus, by~\eqref{eq:c}
\begin{equation}\label{eq:JcapB}
  |\cup_i J^*_{\ell,j_i}|\geq c\,|\cup_i I_{\ell,j_i}|\geq \frac{c}{3} |B_\varepsilon(a_0)|.
\end{equation}
Hence,~\eqref{eq:AcapB} and~\eqref{eq:JcapB} imply\footnote{If $|A\cap B|\geq \Delta |B|$ and $|J|\geq C |B|$ with $J\subset B$, then
$$|A\cap J|=|(A\cap B)\cap (J\cap B)|=|A\cap B| + |J\cap B| - |(A\cap B)\cup (J\cap B)| \geq \Delta |B| + C|B| - |B| =(\Delta+C-1)|B|.$$}
that $$|\cup_i J^*_{\ell,j_i} \cap A| \geq (1-\delta + c/3 -1) |B_\varepsilon(a_0)| \geq \frac{c}{6} |B_\varepsilon(a_0)| \quad \text{for all $\ell$ large enough}$$
if $\delta>0$ is taken close enough to $0$. In particular, denoting by
$$
 J^*_\ell = \{ a\in \mathbb{I}^k: \, g_a \ \text{has a phenomenon
 $\mathscr{P}$ after renormalization of period $n_\ell$}\, \}
$$
and having into account that $J^*_{\ell,j}\subset J^*_\ell$ we get $\varepsilon>0$ such that $|J^*_\ell \cap A| \geq c |B_{\varepsilon}(a_0)|/6$  for every $\ell$ large enough. This implies that
$$\sum_{\ell \geq 1} |A\cap J^*_\ell|=\infty  \quad \text{for all measurable set $A$ with $|A|>0$}.$$
By the generalization of the second Borel-Cantelli lemma in~\cite[Thm.~1(a)]{shuster1970borel},
the set of event that occurs for infinitely many~$\ell$,
that is,
$$
     J^*=\bigcap_{n\geq 1} \bigcup_{\ell\geq n} J^*_{\ell}
$$
has full Lebesgue measure in $\mathbb{I}^k$.
In particular, and since the sequence of period $n_\ell$ is strictly increasing, by~Remark~\ref{rem:geometrically-independence}, for
any $a\in J^*$ the map $g_a$ has coexistence of infinitely many
different phenomena $\mathscr{P}$. This concludes the proof of the
theorem.
\end{proof}

Now we will prove Proposition~\ref{main-lema}. To do this, we need
the following lemma.
%and fix the following notation. Consider a
%point $a_0\in \mathring{\mathbb{I}}^k$ and denote its coordinates
%by $a_0=(a_{01},\dots,a_{0k})$. Moreover, fix $\alpha>0$ so that
%$a_0+(-\alpha,\alpha)^k\subset \mathbb{I}^k$.

\begin{lem} \label{lema2}
Given $\alpha>0$ let $g=(g_a)_a$ be a $C^{d,r}$-family and assume that $g_a$ has a
homoclinic tangency at a point~$Y_a$ (depending $C^d$-continuously
on~$a$) associated with a sectional dissipative saddle $Q_a$ for
any parameter $a\in a_0+(-\alpha,\alpha)^k$. Then, for any
$\rho>0$ and $\kappa>1$, there exists a sequence of families $g_n=(g_{na})_a$
{approaching} $g$ in the $C^{d,r}$-topology such that $g_{na}=g_a$ if
$a\not\in a_0 + (-\kappa\alpha,\kappa\alpha)^k$ and $g_n$ is a
$\rho$-$C^s$-H\'enon-like family after renormalization of period
$n$ in $a_{0} + (-\alpha,\alpha)^k$ for $n$ large enough.
\end{lem}

Before proving this result, let us show how to obtain
Proposition~\ref{main-lema} from~the~above~lemma:

\begin{proof}[Proof of Proposition~\ref{main-lema}]
Consider a family $f=(f_a)_a$ in $\mathcal{D}$. Definition~\ref{def:Berger-domain} provides an open cover of $\mathbb{I}^k$ where $f$ has a persistent homoclinic tangency in each open set of parameters in this cover. Let $L>0$ be a Lebesgue number\footnote{For every open cover $U$ of a compact metric space $X$ there is a positive real number $L$, called a Lebesgue number, such that every subset of $X$ of diameter less than $L$ is contained in some element of $U$.} of this open cover.
Fix $\epsilon>0$ and $\rho>0$ and $0<\alpha \leq \alpha_0=L/2k$. Choose in $\mathbb{I}^k$ finitely many pairwise disjoint open balls of the form $I'_j=a_{j} + (-\kappa_j\alpha,\kappa_j\alpha)^k$ for some $a_j\in \mathbb{I}^k$, $1<\kappa_j\leq (1+\alpha)^{1/k}$ such that the union of these balls is {of full measure} in $\mathbb{I}^k$.
Set $I_j=a_j+(-\alpha,\alpha)^k$. Clearly $|I_j|=\alpha^k|\mathbb{I}^k|\leq\alpha |\mathbb{I}^k|$ (for any $0<\alpha<1$). On the other hand, since $
  |\mathbb{I}^k|=|\cup_j I'_j| \leq (1+\alpha) |\cup_j I_j|$,
we have
$$
 |\mathbb{I}^k\setminus \cup_j I_j| =|\mathbb{I}^k|-|\cup_j I_j|\leq (1-\frac{1}{1+\alpha})|\mathbb{I}^k| \leq \alpha   |\mathbb{I}^k|.
$$
Note that since $\alpha$ is less than $L/2k$, we have that the diameter of any $I_j=a_j+(-\alpha,\alpha)^k$ is smaller than $L$ and thus is contained in one of the open balls of the covering.  Consequently,  $f$ has a persistent homoclinic tangency in $I_j$. That is, $f_a$ has a
homoclinic tangency at a point~$Y_a$ (depending $C^d$-continuously
on~$a$) associated with a sectional dissipative saddle $Q_a$ for
any parameter $a\in I_j$. Taking into
account that $\{I'_j\}_j$ are pairwise
disjoint sets, we can apply Lemma~\ref{lema2} in each of these sets of parameters.  Thus, for any
$n\in\mathbb{N}$ large enough
(depending on $\epsilon$, $\rho$, $f$ and $\alpha$)
we get
an
$\epsilon$-perturbation $g_n=(g_{na})_a$ of $f$ in the
$C^{d,r}$-topology which is a $\rho$-$C^s$-H\'enon-like family
after renormalization of period $n$ in any
$I_j$. This concludes the
proof.
\end{proof}

Finally, to complete the proof, we will prove Lemma~\ref{lema2}.

\begin{proof}[Proof of Lemma~\ref{lema2}]
By assumption, the map $g_a$  has a homoclinic tangency at a point
$Y_a$ associated with a sectional dissipative periodic point $Q_a$
for all $a\in a_0+(-\alpha,\alpha)^k$. Actually, the tangency must
be smoothly continued until $\|a-a_0\|_{\infty}=\alpha$. By means
of an arbitrarily small $C^{d,r}$-perturbation of the family
around the tangency point, we can assume that the tangency $Y_a$
is simple in the sense of~\cite{GST08}. That is, the tangency is
quadratic, of codimension one, and, in the case of dimension
$m>3$, any extended unstable manifold is transverse to the leaf of
the strong stable foliation that passes through the tangency
point. Since $Q_a$ is a sectional dissipative saddle, if we denote
the leading multipliers of this periodic point by $\lambda_a$ and
$\gamma_a\in \mathbb{R}$ we have
\begin{equation*}
      |\lambda_a|<1< |\gamma_a|  \quad \text{and}  \quad |\lambda^{}_a
      \gamma_a|<1.
\end{equation*}

On the other hand, we can consider a generic one-parameter
unfolding $g_{a,\mu}$ of the homoclinic tangency of $g_a$. To be
more specific, we consider the one-parameter unfolding $g_{a,\mu}$
of $g_a$ where $\mu$ is the parameter that controls the splitting
of the tangency. We can take local coordinates $(x,y)$  with $x\in
\mathbb{R}^{m-1}$ and $y\in \mathbb{R}$ in a neighborhood of $Q_a$
which corresponds to the origin such that $W^s_{loc}(Q_a)$ and
$W^u_{loc}(Q_a)$ acquire the form $\{y=0\}$ and $\{x=0\}$
respectively. Moreover,  by considering, if necessary,
iterated, the tangency point $Y_a$ is represented by $(x^+,0)$ in
these coordinates. Let us consider a $C^\infty$-bump function $
\phi:\mathbb{R}\to \mathbb{R}$ with support in $[-1,1]$ and equal
to 1 on $[-1/\kappa,1/\kappa]$. Let $$\varphi: a=(a_1,\dots,a_k) \in
\mathbb{I}^k \mapsto \phi(a_1)\cdot\ldots\cdot \phi(a_k) \in
\mathbb{R}.$$ Take $\delta>0$ such that the
$2\delta$-neighborhoods in local coordinates of $Y_a$, $g_a(Y_a)$
and $g_a^{-1}(Y_a)$ are disjoint. In particular, we denote by $U$
the $2\delta$-neighborhood of $Y_a$ in these local coordinates.
%which from the smoothness with respect to $a$ of the tangency point
% can be taken independent of $a$.
We  write
$$g_{a,\mu}= H_{a,\mu}
\circ g_a$$ where $H_{a,{\mu}}$  in this local coordinates takes
the form
\begin{align*}
   \bar{x}&= x
   \\
   \bar{y}&=y+\varphi\left(\frac{a-a_0}{\kappa\alpha}\right)\phi\left(\frac{\|(x,y)-(x^+,0)\|}{2\delta}\right)\,\mu
\end{align*}
and it is the identity otherwise. Observe that if $a\not\in a_0 +
(-\kappa\alpha,\kappa\alpha)^k$ then $g_{a,\mu}=g_a$. Also, if $(x,y)\not \in
g_a^{-1}(U)$ then $g_{a,\mu}=g_a$.

Let us define the first return map associated with the simple
homoclinic tangency of $g_{a,\mu}$ at $\mu=0$ (see~\cite[Sec.~1,
p.~928]{GST08}). As usual, $T_0=T_0(a,\mu)$ denotes the local map
for $a\in a_0 + [-\alpha,\alpha]^k$. In our case, since we are assuming that $Q_a$ is a fixed point, $T_0$
corresponds to $g_{a,\mu}$ at a neighborhood of this fixed point. By $T_1=T_1(a,\mu)$ we denote the map
$g_{a,\mu}^{k_1}$ from a neighborhood $\Pi^-_a$ of a tangent point
$Y^-_a\in W^u_{loc}(Q_a)$  to a neighborhood $\Pi_a$ of
$Y_a=g_a^{k_1}(Y^-_a)\in W^s_{loc}(Q_a)$. Then, one defines the
first-return map as $T_n=T_1\circ T_0^n$ for sufficiently large
$n$ on $\sigma_n(a)=\Pi_a \cap T_0^{-n}(\Pi^-_a)$. Since the
tangency point $Y_a$ depends $C^d$-continuously on
$a_0+[-\alpha,\alpha]^k$, we find that this first-return map
$T_n=T_n(a,\mu)$ also depends smoothly as a function of the
parameter $a$ on $a_0+[-\alpha,\alpha]^k$.

\begin{lem}[Parametrized rescaling lemma] \label{lem.new}
There are a sequence of open sets $\Delta_n$  in the
\mbox{$(a,\mu)$-parameter} space with $\overline{\Delta_n}$
accumulating on $(a_0+[-\alpha,\alpha]^k)\times \{0\}$ such that
for any $(a,\mu)\in \overline{\Delta_n}$ there is
 a smooth transformation of coordinates which
brings the first-return map $T_n$ in local coordinates on
$\sigma_n(a)$ the following form:
\begin{equation*} %\label{eq-GHM2}
   \bar{x}=o(1)  \quad \text{and} \quad  \bar{y}=M-y^2+ o(1)
\end{equation*}
where the $o(1)$-terms tends to zero (uniformly on $a$) as $n\to
\infty$ along with all the derivatives up to the order $r$ with
respect to the coordinates $(x,y)$ and up to $d\leq r-2$ with
respect to the rescaled parameter~$M$. The domain of definition of
$T_n$ in these coordinates is an asymptotically large region
that, as $n\to \infty$, covers all finite values of $(x,y)$. The
rescaled parameter $M$ is at least $C^d$-smooth function of
$(a,\mu)$ which for large enough $n$ is given by
\begin{align}  \label{eq:MB} M \sim
\gamma^{2n}_a(\mu + O(\gamma_a^{-n})).
\end{align}
\end{lem}
\begin{proof}  {This result follows from~\cite[Lemma~1]{GST08}. To see this, let us analyze the proof of the rescaling lemma in~\cite{GST08} for the sectional dissipative case $(1,1)$. We observe that we can  perform the proof line by line and the transformation of coordinates~\cite[Eq.~(3.4) and (3.7)]{GST08} can be done smoothly on the parameter $a\in a_0+[-\alpha,\alpha]^k$. On the other hand, the constants on the $O$-terms will depend on the parameter $a$ but these can be uniformly  bounded due to the compactness of the parameter space $a_0+[-\alpha,\alpha]^k$ and the continuity of all the coefficients with respect to $a$.}
%
%
%
%A carefully reader of the proof of the rescaling lemmas
%in~\cite[Lemma~1 and Lemma~4]{GST08} allows us to apply these
%results smoothly on the parameter $a \in a_0+[-\alpha,\alpha]^k$.
%That is, all the expression and change of variables
%in~\cite[Lemma~1]{GST08} can be performed smoothly on $a\in
%a_0+[-\alpha,\alpha]^k$. This immediately implies the lemma states
%above.
%%and, from the compactness of the parameter
%%space, the corresponding $O$-terms that appear after rescaling can
%%be estimated uniform also on $a$.
%%Applying the rescaling lemmas in~\cite[Lemma~1 and~4]{GST08}
\end{proof}
Notice that the parameter $M$ in~\eqref{eq:MB} can take arbitrarily finite values when $\mu$ varies
close to $\mu^0_n(a)=O(\gamma^{-n}_a)$. To be more precisely, the
parameter $\mu_n^0(a)$ was introduced in~\cite{GST08} so that
$M_{n,a}(\mu^0_n(a))=0$ where $M_{n,a}$ is the function given
in~\eqref{eq:MB} for fixed $n$ and $a$. Actually,  an explicit
expression of $M$ in~\eqref{eq:MB} is provided in~\cite[after
Eq.~(3.8)]{GST08}  which, up to multiplicative constants, is
basically the right hand of~\eqref{eq:MB} where the $O$-function
does not depend on $\mu$ and its $i$-th partial derivatives with
respect to the variable $a$ are of order $O(n^i\gamma^{-n}_a)$.
Thus, we can calculate the derivative with respect to $\mu$ of
$M_{n,a}$ for $n$ large enough as $
 \partial_\mu M_{n,a} \sim \gamma_a^{2n} \gg 1
$. Hence, we obtain that $M_{n,a}$  is an invertible expanding map with
an arbitrarily large uniform expansion on $a\in a_0 +
[-\alpha,\alpha]^k$. Thus, for $n$ large enough, we can assume
that
$$
\Phi_n(a,\mu)=(a, M_{n,a}(\mu))
$$
is a diffeomorphism between the set $\overline{\Delta_n}$ given
above in the lemma and $(a_0+[-\alpha,\alpha]^k)\times [-10,10]$.

Notice that the linear rescaling $b(M)$ given in
Definition~\ref{def:Henon-like-after} takes, on
$a_{0}+[-\alpha,\alpha]^k$, the form
$$
   b(M)=2\alpha M -3\alpha+\pi_k(a_{0}) \quad \text{for} \ \ M\in
    [1,2]
$$
where $\pi_k: \mathbb{R}^k\to \mathbb{R}$ is the projection on the
$k$-th coordinate. Consider the inverse map
$$
   \widehat{M}(b)= \frac{b+3\alpha-\pi_k(a_{0})}{2\alpha}  \quad \text{for} \ \ b\in
   \pi_k\left(a_{0}
   +[-\alpha,-\alpha]^k\right).
$$
Now  since $\Phi_n$ is a diffeomorphism, we find a
$C^{d}$-function $\mu_n$ on $a_0+[-\alpha,\alpha]^k$ defined as
$$
\Phi_n^{-1}\left(a,\widehat{M}(\pi_k(a))\right)=(a,\mu_n(a))
\qquad a\in a_0+[-\alpha,\alpha]^k.
$$
%where  $a=(\bar{a},b)\in \mathbb{I}^{k-1}\times(a_{0k}+(-\alpha,\alpha))$.
In particular,
\begin{equation} \label{eq2}
\begin{aligned}
  %M(a,\mu_{n}(a))=
  M_{n,a}(\mu_n(a))=\widehat{M}(\pi_k(a)) \quad \text{for} \  \ a\in
  a_{0}+(-\alpha,\alpha)^k.
\end{aligned}
\end{equation}
Extending smoothly $\mu_n$ to $\mathbb{I}^k$
(c.f.~\cite[Lemma~2.26]{lee2012introduction}) we can consider the
sequence of families ${g}_n=({g}_{n,a})_a$ where
$$
   {g}_{n,a}=g_{a,\mu_n(a)} \qquad \text{for  $a\in
   \mathbb{I}^k$ and $n$ large enough}.
$$
Observe that ${g}_{n,a}=g_a$ for $a \not \in a_0 +
(-\kappa\alpha,\kappa\alpha)^k$. Moreover, according to Lemma~\ref{lem.new}
and Equation~\eqref{eq2}, there is
 a smooth family $R=(R_a)_a$ of smooth transformation of coordinates $R_{a}$
 on $\sigma_n(a)$ such that bring the first-return map $T_n=T_n(a,\mu_n(a))$ of $g_{n,a}$
 into $R_a^{-1}\circ T^{}_n
 \circ R^{}_a$ which has the form
$$
   \bar{x}=o(1)  \quad \text{and} \quad  \bar{y}=
   \widehat M(\pi_k(a))-y^2+o(1)  \qquad  \text{for \ } a \in a_0 + [-\alpha,\alpha]^k.
$$
Substituting  the parameter $a$ by the linear rescaling $a(M)=(\bar{a},b(M))\in
a_0 + [-\alpha,\alpha]^k$ and taking into account that
$\widehat{M}(\pi_k(a(M)))=\widehat{M}(b(M))=M$ for all $M\in
[1,2]$ we obtain that $\varphi_n=R^{-1}_{a(M)}\circ T^{}_n(a(M),\mu_n(a(M)))
\circ R^{}_{a(M)}$ %in local coordinates
takes the form
$$
   \bar{x}=o(1)  \quad \text{and} \quad  \bar{y}=M-y^2+ o(1) \quad
   \text{for} \ \ M\in [1,2].
$$
Since the $o(1)$-terms above tend to zero as $n\to \infty$ along
with all derivatives up to the order $r$ with respect to the
coordinates $(x,y)$ and up to $s\leq d\leq r-2$ with respect to
$M$, we get $\|\varphi_n-\Phi\|_{C^{s,s+2}}=o(1)$ where
$\Phi=(\Phi_M)_M$ is the parabola family.  This proves that for
$n$ large enough, $g_n$ is a $\rho$-$C^s$-H\'enon-like family
after renormalization of period $\tilde{n}=n + k_1$ in
$a_0+(-\alpha,\alpha)^k$. For short and simplicity, we can relabel the sequence of families
to simply say that the renormalization period of $g_n$ is $n$.

To conclude the proof of the lemma, we only need to show that
${g}_n$ converges to $g$ in the $C^{d,r}$-topology. To do this,
notice that the $C^{d,r}$-norm
$$
 \|{g}_n-g\|=\|(I-H_{a,\mu_n(a)})\circ g_a \|  \leq
  \|I-H_{a,\mu_n(a)}\| \, \|g\|
$$
where $I$ denotes the identity. Thus, we only need to calculate
the $C^{d,r}$-norm of the family $(I-H_{a,\mu_n(a)})_a$. Since
$H_{a,\mu_n(a)}=I$ if $a\not \in a_0 + (-\kappa\alpha,\kappa\alpha)^k$ or
$(x,y)\not \in U $ then
\begin{align*}
   \left\|I-H_{a,\mu_n(a)}\right\|\leq \left\|\varphi\left(\frac{a-a_0}{\kappa\alpha}\right)\phi\left(\frac{\|(x,y)-(x^+,0)\|}{2\delta}\right)\,\mu_n(a)\right\|.
\end{align*}
Then to estimate the $C^{d,r}$-norm above it suffices to show that
the function
$$
   G_{n,\alpha}(a)  = \varphi(\frac{a-a_0}{\kappa\alpha})\mu_n(a)  \quad
   \text{for $a \in  a_0+(-\kappa\alpha,\kappa\alpha)^k$}
$$
have $C^d$-norm small when $n$ is large. To do this, using the
multi-index notation for partial derivatives and the Leibnitz rule,
\begin{equation} \label{above}
\partial_a^\ell G_{n,\alpha}(a)=\sum_{j\leq \ell} \binom{\ell}{j} \,
\partial_a^{\ell-j} \varphi\left(\frac{a-a_0}{\kappa\alpha}\right) \cdot
\partial_a^j\mu_n(a) \qquad  \ell\in \mathbb{Z}_+^k=(\mathbb{N}\cup \{0\})^k \ \ \text{with} \ \
|\ell|\leq d.
\end{equation}
On the other hand, recall that as we have indicated before, the
$\mu$ variable in~\eqref{eq:MB} varies close to
$\mu_n^0(a)=O(\gamma_a^{-n})$.
%which satisfies that $\partial_a^j\mu_n^0(a)=O(n^{|j|}\gamma_a^{-n})$ for all
%$j\geq 0$. Moreover, $|\partial_a^j\mu_n(a)-\partial_a^j\mu_n^0(a)|\to 0$ as
%$n\to\infty$. Thus, for $n$ large enough we also get  that
%$\partial_a^j\mu_n(a)=O(n^{|j|} \gamma_a^{-n})$ for $j\geq 0$.
Moreover, $|\mu_n(a)-\mu_n^0(a)|\to 0$ as $n\to\infty$. Thus, we
also get $\mu_n(a)=O(\gamma_a^{-n})$.

We will show that
$\partial_a^\ell\mu_n(a)=O(n^{|\ell|}\gamma^{-n}_a)$ for $\ell\in
\mathbb{Z}^k_+$ by induction in $|\ell|$. To do this, assume that
$\partial_a^j\mu_n(a)=O(n^{|j|}\gamma_a^{-n})$ for all
$j\in\mathbb{Z}_+^k$ with $j \leq \ell$ and $|j|<|\ell|$. Notice
that~\eqref{eq2} can be written as
$$
    \widehat M(\pi_k(a)) \sim
    \gamma^{2n}_a(\mu_n(a)+O(\gamma_a^{-n}))
$$
where the equivalence is actually an equality up to multiplicative
constants (independent of~$a$). Moreover, the $i$-th partial derivatives
with respect to the variable $a$ of the above $O$-function is of order
$O(n^i\gamma^{-n}_a)$ for all $i\geq 0$. Thus,  we get that
\begin{align*}
 \partial_a^\ell \widehat M(\pi_k(a))=O(1) &\sim O(n^{|\ell|}\gamma_a^{n}) + \sum_{j\leq
 \ell} \binom{\ell}{j} \, \partial_a^{\ell-j}(\gamma_a^{2n})\cdot
 \partial_a^j\mu_n(a) \\
 &=O(n^{|\ell|}\gamma_a^{n}) + \gamma^{2n}_a \cdot \partial_a^\ell\mu_n(a) +  \sum_{j\leq
 \ell \, j\not =\ell} O(n^{|\ell|-|j|}\gamma_a^{2n}) \cdot
 O(n^{|j|}\gamma_a^{-n}) \\
 &= O(n^{|\ell|}\gamma_a^{n}) + \gamma^{2n}_a \cdot
 \partial_a^\ell\mu_n(a).
\end{align*}
From here it follows that
$\partial_a^\ell\mu_n(a)=O(n^{|\ell|}\gamma_a^{-n})$. Indeed, if
$\partial_a^\ell\mu_n(a)$ is not a $O(n^{|\ell|}\gamma_a^{-n})$,
then negating the definition of $O$-function, for every $K>0$ and
$n_0\in\mathbb{N}$ there exists $n\geq n_0$ such that
$|\partial_a^\ell\mu_n(a)|>K n^{|\ell|}|\gamma_a|^{-n}$. In
particular,
$\partial_a^\ell\mu_n(a)=\Omega(n^{|\ell|}\gamma_a^{-n})$ where
$\Omega$ denotes the Big Omega of Hardy-Littlewood. Hence, we
obtain that $\partial_a^\ell M(\pi_k(a))$, which is a
$O(1)$-function, is also of order $O(n^{|\ell|}\gamma_a^{n}) +
\gamma^{2n}_a \cdot \Omega( n^{|\ell|}\gamma_a^{-n})=
O(n^{|\ell|}\gamma_a^{n}) +  \Omega(
n^{|\ell|}\gamma_a^{n})=\Omega( n^{|\ell|}\gamma_a^{n})$ obtaining
a contradiction.

Substituting the above estimate of $\partial_a^\ell\mu_n(a)$ into~\eqref{above}, we
get that $\partial_a^\ell G_{n,\alpha}(a)= O((\kappa\alpha)^{-d}n^{|\ell|}
\gamma_a^{-n})$. In particular, we get
\begin{equation*}
\|G_{n,\alpha}\|_{C^d}= O\left((\kappa\alpha)^{-d}n^d \gamma^{-n}\right)
\end{equation*}
for some $1<\gamma\leq \gamma_a$ for all $a\in a_0+
(-\alpha,\alpha)^k$. Observe that this assertion completes the
proof of the lemma.
\end{proof}

\subsection*{Acknowledgements}
We thank E.~R.~Pujals for his guidance, and the encouragement he gave us to write this paper providing many ideas to go ahead. The second author
also especially thanks to his supervisor E.~R.~Pujals for his
unconditional friendship and enriching talks on mathematics among
other things during his doctorate. Finally, the first author
thanks A.~Raibekas for his tireless patience and friendship during
many difficult moments throughout the process of writing and
revising the preliminary versions.

The first author was partially supported by CNPq, FAPERJ and
grants MTM2017-87697-P and PID2020-113052GB-I00 funded by MCIN (Spain).
\bibliographystyle{alpha2}
\bibliography{bib}

\end{document}